\documentclass[10pt]{amsart}
\usepackage[all]{xy}
\usepackage{amsfonts,amssymb,amscd,amsmath,enumerate,verbatim,calc}
\usepackage{mathrsfs}
\numberwithin{equation}{section}
\newtheorem{theorem}{Theorem}[section]
\newtheorem{corollary}[theorem]{Corollary}

\newtheorem{proposition}[theorem]{Proposition}
\newtheorem{definition}[theorem]{Definition}

\newtheorem{example}[theorem]{Example}
\newtheorem{remark}[theorem]{Remark}

\begin{document}

\title[Twisted Lie algebras by invertible derivations]
 {Twisted Lie algebras by invertible derivations}

\bibliographystyle{amsplain}

\author[I. Basdouri, E. Peyghan and M.A. Sadraoui]{Imed Basdouri, Esmael Peyghan and Mohamed Amin Sadraoui}
\address{University of Gafsa, Faculty of Sciences Gafsa, 2112 Gafsa, Tunisia.}
\email{basdourimed@yahoo. fr}
\address{Department of Mathematics, Faculty of Science, Arak University,
	Arak, 38156-8-8349, Iran.}
\email{e-peyghan@araku.ac.ir}
\address{University of Sfax, Faculty of Sciences Sfax, BP
	1171, 3038 Sfax, Tunisia.}
\email{aminsadrawi@gmail. com}


\subjclass[2010]{}


\begin{abstract}
In this paper, we introduce an algebra structure denoted by InvDer algebra whose which we twist an algebra thanks to an invertible derivation, where its inverse is also a derivation. We define InvDer Lie algebras, InvDer associated algebras, InvDer zinbiel algebras and InvDer dendriforme algebras. We also study the relations between these structures using the Rota-Baxter operators and the endomorphism operators.

\end{abstract}

\maketitle

\section{Introduction}
A Lie algebra, named for Sophus Lie, is a vector space $L$ endowed with a skew symmetric bilinear map $[\cdot,\cdot]:L\otimes L \rightarrow L$ which satisfies the Jacobi identity
\begin{equation*}
\circlearrowleft_{x,y,z}[x,[y,z]]=[x,[y,z]]+[y,[z,x]]+[y,[z,x]]=0, \ \ \ \forall x,y,z \in L.
\end{equation*}
The study of these algebras goes back to the end of the nineteenth century when Sophus Lie investigated symmetry groups of differential equations. This investigations leads to study locally the Lie algebra structure on the tangent space in the identity element of the group. Lie algebras are at the core of several areas of mathematics, for examples, algebraic groups, quantum groups, representation theory, homogeneous spaces, integrable systems and algebraic topology.

Many researchers continue to study all its topics until the rises of Hom-Lie algebras in 2006, when J. T. Hartwig, D. Larrsson and S. D. Silvestrov define the notion of Hom-Lie algebras (see \cite{H1} for more details), using an homomorphism $\zeta$, where the Jacobi identity becomes 
\begin{equation*}
\circlearrowleft_{x,y,z}[(id+\zeta)(x),[y,z]]=0.
\end{equation*}

In 2010, A. Makhlouf and S. D. Silvestrov studied more precisely the notion of Hom-Lie algebras (see \cite{M1} for more details), where they changed $(id+\zeta)$ to be an homomorphism $\alpha$ and the Jacobi identity becomes
\begin{equation*}
\circlearrowleft_{x,y,z}[\alpha(x),[y,z]]=0.
\end{equation*}

Y. Donald in \cite{Y1}, gives the notion of twisted Hom algebras, the idea was to construct an Hom algebra $(L,\mu_{\alpha},\alpha)$ from an algebra $(L,\mu)$ where $\mu_{\alpha}=\alpha \circ \mu$ (see Theorem 2.4 in \cite{Y1}). 

Among the studies of Lie algebras, the derivation of a Lie algebra is investigated too. Let $(L,[\cdot,\cdot])$ be a Lie algebra and $\delta :L \rightarrow L$ be a linear map. $\delta$ is said to be a derivation on $(L,[\cdot,\cdot])$, if it satisfies 
\begin{equation*}
\delta([x,y])=[\delta x,y]+[x,\delta y] \, , \ \ \forall x,y \in L.
\end{equation*}

Many results are obtained from the concept of derivation, for instance in \cite{T1} the authors construct a new algebra via its derivation denoted by LieDer pair and study its cohomology. Also in \cite{C1} the authors characterized the space of all derivations of Lie algebras. the space of all derivations of a Lie algebra $L$ is denoted by \textbf{Der(L)}. From this space rises a Lie algebra  via the following bilinear map 
\begin{equation*}
[\delta_1,\delta_2]=\delta_1 \circ \delta_2-\delta_2 \circ \delta_1, \ \ \ \forall \delta_1,\delta_2 \in \textbf{Der(L)}.
\end{equation*}
Note that if $\delta$ is a derivation on $L$, then $\delta^2$ is not a derivation because the Leibniz rule does not hold
\begin{equation} \label{eq1.1}
\delta^2([x,y])=[\delta^2x,y]+[x,\delta^2y]+2[\delta x,\delta y], \ \ \ \ \forall x,y \in L.
\end{equation}

Inspired by the Yau's twist \cite{Y1} and the identity \eqref{eq1.1}, we have the idea to twist a Lie algebra $(L,[\cdot,\cdot])$ to a new one which its structure is characterized by $[\cdot,\cdot]_\delta=\delta \circ [\cdot,\cdot]$ and we get the following equation,
\begin{equation} \label{eq1.2}
\circlearrowleft_{x,y,z}[x,[y,z]_{\delta}]_{\delta}=\circlearrowleft_{x,y,z}[\delta x,[\delta y,z]]+\circlearrowleft_{x,y,z}[\delta x,[y,\delta z]]+\circlearrowleft_{x,y,z}[x,\delta^2[y,z]],
\end{equation} 
where for all $x,y,z \in L$
\begin{eqnarray*}
	\circlearrowleft_{x,y,z}[\delta x,[\delta y,z]]&=&[\delta x,[\delta y,z]]+[\delta y,[\delta z,x]]+[\delta z,[\delta x,y]], \\
	\circlearrowleft_{x,y,z}[\delta x,[y,\delta z]]&=&[\delta x,[y,\delta z]]+[\delta y,[z,\delta x]]+[\delta z,[x,\delta y]], \\
	\circlearrowleft_{x,y,z}[x,\delta^2[y,z]]&=&[x,\delta^2[y,z]]+[y,\delta^2[z,x]]+[z,\delta^2[x,y]].
\end{eqnarray*}
Unfortunately, the equation \eqref{eq1.2} does not lead to the Jacobi identity. So the main idea was to add a condition on the derivation $\delta$ to let us obtain the Jacobi identity which is mentioned in section \ref{sec1}.

The paper is organized as follows: In Section 2, we introduce the notion of InvDer Lie and pre Lie algebras thanks to an invertible derivation whose inverse is also a derivation by giving some examples. in Section 3, we study two passages between these two algebraic structures so we introduce the InvDer associative algebras by linking it with the InvDer Lie and pre Lie algebras. Finally in Section 4, we introduce the InvDer zinbiel and dendriform algebras.

Throughout this paper we consider $\mathbb{K}$ as a field of characteristic $0$ and all algebras are defined over $\mathbb{K}$.
\section{Lie and pre Lie algebras twisted by invertible derivation}
\label{sec1}

Recall first that if $\delta$ is a derivation on a Lie algebra $(L,[\cdot,\cdot])$ then we have the following equation 
\begin{equation*} 
\delta^2 [x,y]=[\delta^2x,y]+[x,\delta^2y]+2[\delta x,\delta y] \ \ \forall x,y \in L,
\end{equation*}
which means $\delta^2$ is not a derivation. 

In this section and from equation \eqref{eq1.1} we have the idea of using an invertible derivation $\delta$ to introduce a new algebra (respectively, pre-Lie algebra) $(L,[\cdot,\cdot]_{ \delta}=\delta \circ [\cdot,\cdot])$ (respectively, $(L,\ast_\delta=\delta \circ \ast)$) from $(L,[\cdot,\cdot]$) with a require condition which we mentioned it in Theorem \ref{thm2.1}. Let's first prepare, in the next proposition, the necessary tools that we need later.
\begin{proposition} \label{prop2.1}
	Let $(L,[\cdot,\cdot])$ be a Lie algebra, and $\delta$ is an invertible derivation. Then the following identity holds 
	\begin{equation} \label{eq2.2}
	[\delta x,\delta y]=\delta^2 [x,y] \Leftrightarrow \delta^{-1} \text{is a derivation}.
	\end{equation}
\end{proposition}
\begin{proof}
	Let $x,y \in L$. For the first sense, let's suppose that  $[\delta x,\delta y]=\delta^2 [x,y]$ and prove that $\delta^{-1}$ is a derivation. Using \eqref{eq2.2}, we get
	\begin{align*}
	\delta^{-1}[x,y]=\delta^{-1}[\delta \circ \delta^{-1} x,\delta \circ \delta^{-1} y]=\delta^{-1}\circ  \delta^2  [\delta^{-1}x,\delta^{-1}y]=\delta [\delta^{-1}x,\delta^{-1}y]=[\delta^{-1}x,y]+[x,\delta^{-1}y],
	\end{align*}
	which means $\delta^{-1}$ is a derivation. Conversely, we suppose that $\delta^{-1}$ is a derivation and prove that $[\delta x,\delta y]=\delta^2 [x,y]$. Direct calculations give us
	\begin{align*}
	\delta^2[x,y]=\delta ( \delta [x,y]) =\delta ([\delta x,y]+[x,\delta y] )=\delta ([\delta x,\delta^{-1} \circ \delta y]+[\delta^{-1} \circ \delta x,\delta y]) = \delta ( \delta^{-1}[\delta x,\delta y])=[\delta x,\delta y].
	\end{align*}
	This completes the proof.
\end{proof} 
\begin{remark}
	For the second sense there is another manner to prove that $[\delta x,\delta y]=\delta^2 [x,y]$. Using equation \eqref{eq1.1} we get 
	\begin{align*}
	\delta^2 [x,y]&=[\delta^2x,y]+[x,\delta^2y]+2[\delta x,\delta y] \\
	&=[\delta^2x,\delta^{-1} \circ \delta y]+[\delta^{-1} \circ \delta x,\delta^2y]+[\delta^{-1} \circ \delta^2 x,\delta y]+[\delta x,\delta^{-1} \circ \delta^2 y] \\
	&=\delta^{-1}\Big([\delta^2x,\delta y] \Big)+\delta^{-1}\Big([\delta x,\delta^2 y] \Big) \\
	&= \delta^{-1} \Big(\delta[\delta x,\delta y] \Big) \\
	&=[\delta x,\delta y].
	\end{align*}
\end{remark}
In the next theorem, inspired by \cite{Y1}, we present a Lie algebra structure by an invertible derivation such that its inverse is also a derivation. 
\begin{theorem}\label{thm2.1}
	Let $(L,[\cdot,\cdot])$ be a Lie algebra and $\delta:L \rightarrow L$ be an invertible derivation such that its inverse is a derivation. Then $(L,[\cdot,\cdot]_{\delta}=\delta \circ [\cdot,\cdot])$ is a Lie algebra, which is called \textbf{ twisted Lie algebra by an invertible derivation }. 
\end{theorem}
\begin{proof}
	Let $x,y,z \in L$. We have
	$$[x,y]_{\delta}=\delta \circ [x,y]=-\delta \circ [y,x]=-[y,x]_{\delta},$$
	which means that the skew symmetry property holds. Now we check Jacobi identity. We have
	\begin{align*}
	[x,[y,z]_{\delta}]_{\delta}&=\delta \circ [x,[y,z]_{\delta}] \\
	&=[\delta x,[y,z]_{\delta}]+[x,\delta [y,z]_{\delta}] \\
	&=[\delta x,[\delta y,z]]+[\delta x,[y,\delta z]]+[x,\delta^2[y,z]],
	\end{align*}
	which implies
	$$[x,[y,z]_{\delta}]_{\delta}=[\delta x,[\delta y,z]]+[\delta x,[y,\delta z]]+[x,\delta^2[y,z]].$$
	Similarly, we obtain 
	\begin{align*}
	[y,[z,x]_{\delta}]_{\delta}&=[\delta y,[\delta z,x]]+[\delta y,[z,\delta x]]+[y,\delta^2[z,x]], \\
	[z,[x,y]_{\delta}]_{\delta}&=[\delta z,[\delta x,y]]+[\delta z,[x,\delta y]]+[z,\delta^2[x,y]].
	\end{align*}
	So 
	\begin{align*}
	\circlearrowleft_{x,y,z}[x,[y,z]_{\delta}]_{\delta}&=[\delta x,[\delta y,z]]+[\delta y,[z,\delta x]]+[\delta x,[y,\delta z]]+[\delta z,[\delta x,y]]\\
	&\ \ +[\delta y,[\delta z,x]]+[\delta z,[x,\delta y]]+[x,\delta^2[y,z]]+[y,\delta^2[z,x]]+[z,\delta^2[x,y]].
	\end{align*}
	Using the Jacobi identity of $(L,[\cdot,\cdot])$, we obtain 
	\begin{align*}
	\circlearrowleft_{x,y,z}[x,[y,z]_{\delta}]_{\delta}&=[[\delta x,\delta y],z]+[y,[\delta x,\delta z]]+[[\delta y,\delta z],x]+[x,\delta^2[y,z]]	+[y,\delta^2[z,x]]+[z,\delta^2[x,y]] \\
	&=-[z,[\delta x,\delta y]]-	[y,[\delta z,\delta x]]-[x,[\delta y,\delta z]]+[x,\delta^2[y,z]]	+[y,\delta^2[z,x]]+[z,\delta^2[x,y]].
	\end{align*} 
	Setting \eqref{eq2.2} in Proposition \ref{prop2.1}, we obtain 
	\begin{equation*} 
	\circlearrowleft_{x,y,z}[x,[y,z]_{\delta}]_{\delta}=-\circlearrowleft_{x,y,z}[x,\delta^2[y,z]]+\circlearrowleft_{x,y,z}[x,\delta^2[y,z]],
	\end{equation*}
	which gives
	\begin{equation} \label{eq2.3}
	\circlearrowleft_{x,y,z}[x,[y,z]_{\delta}]_{\delta}=0.
	\end{equation}
	So, the Jacobi identity holds. This completes the proof.
\end{proof}
In the following example, we show in the definition of a twisted Lie algebra, the inverse of $\delta$ must be derivative.
\begin{example}
	Let $(L,[\cdot,\cdot])$ be a Lie algebra and $\delta :L \rightarrow L$ be an invertible derivation. Defining a Lie bracket by
	$$[\cdot,\cdot]_{\delta}=\delta \circ [\cdot,\cdot],$$
 we get
	$$[x,[y,z]_{\delta}]_{\delta}=\delta \circ [x,[y,z]_{\delta}]
	=[\delta x,\delta \circ [y,z]]+[x,\delta \circ [y,z]],$$ 
	which implies
	$$[x,[y,z]_{\delta}]_{\delta}=[\delta x,\delta \circ [y,z]]+[x,\delta \circ [y,z]].$$
	Similarly, we obtain
	$$[y,[z,x]_{\delta}]_{\delta}=[\delta y,\delta \circ [z,x]]+[y,\delta \circ [z,x]],$$
	and
	$$[z,[x,y]_{\delta}]_{\delta}=[\delta z,\delta \circ [x,y]]+[z,\delta \circ [x,y]].$$
Using the above equations we obtain
	$\circlearrowleft_{x,y,z}[x,[y,z]_{\delta}]_{\delta}\neq 0$, 
	which implies that $(L,[\cdot,\cdot]_{\delta})$ is not a Lie algebra.	
\end{example}
Here, we present an example of a twisted Lie algebra.
\begin{example}
	Let $L$ be a $3$-dimensional Lie algebras of basis $\{e_1,e_2,e_3\}$ with
	\begin{equation*}
		[e_1,e_2]=e_3, \ \ \ [e_2,e_3]=e_1,\ \ \  [e_3,e_1]=e_2,
	\end{equation*} 
	and $\delta :L \rightarrow L$ be an invertible derivation such that its inverse is a derivation, where
	\begin{align*}
	\delta(e_1)=a_1e_1+a_2e_2+a_3e_3,\ \ \ \delta(e_2)=b_1e_1+b_2e_2+b_3e_3,\ \ \ \delta(e_1)=c_1e_1+c_2e_2+c_3e_3.
	\end{align*}
	Then $(L,[\cdot,\cdot]_{\delta})$ is a twisted Lie algebra. Indeed, we have
	\begin{align*}
	[e_1,e_2]_{\delta}&=-a_3e_1-b_3e_2+(a_1+b_2)e_3,\ \ \ [e_2,e_3]_{\delta}=(b_2+c_3)e_1-b_1e_2-c_1e_3,\\
	[e_3,e_1]_{\delta}&=-a_2e_1+(a_1+c_3)e_2-c_2e_3.
	\end{align*}
	So, we obtain
	\begin{align*}
	[e_1,[e_2,e_3]_{\delta}]_{\delta}&=-b_1(-a_3e_1-b_3e_2+(a_1+b_2)e_3)+c_1(-a_2e_1+(a_1+c_3)e_2-c_2e_3),\\
	[e_2,[e_3,e_1]_{\delta}]_{\delta}&=a_2(-a_3e_1-b_3e_2+(a_1+b_2)e_3)-c_2((b_2+c_3)e_1-b_1e_2-c_1e_3),\\
	[e_3,[e_1,e_2]_{\delta}]_{\delta}&=-a_3(-a_2e_1+(a_1+c_3)e_2-c_2e_3)+b_3((b_2+c_3)e_1-b_1e_2-c_1e_3).
	\end{align*}
	These calculations lead to verify the Jacoby identity.
\end{example}
Theorem \ref{thm2.1} says that a Lie algebra deform into a new Lie algebra via an invertible derivation such that its inverse is also a derivation. Note that the Lie bracket $[\cdot,\cdot]_{\delta}=\delta \circ [\cdot,\cdot]$ is an inspiration of \cite{Y1}, where the twist map is a linear map $\alpha$ such that $[\cdot,\cdot]_{\alpha}=\alpha \circ [\cdot,\cdot]$.
\begin{remark}
	Let's focus on the equation \eqref{eq2.3}. Since $\delta$ is invertible, then we have $\circlearrowleft_{x,y,z}[x,[y,z]_{\delta}]_{\delta}=0$, and consequently $\delta ( \circlearrowleft_{x,y,z}[x,[y,z]_{\delta}])=0$. Since $\ker(\delta)=0$, so we derive that 
	\begin{equation} \label{eq2.4}
	\circlearrowleft_{x,y,z}[x,[y,z]_{\delta}]=0.
	\end{equation}
\end{remark}
This remark allows us to announce the next proposition.
\begin{proposition} \label{prop2.2}
	Let $(L,[\cdot,\cdot])$ be a Lie algebra and $\delta$ be an invertible derivation such that its inverse is a derivation on $[\cdot,\cdot]$. Then, the following equation holds
	\begin{equation} \label{eq2.5}
	\circlearrowleft_{x,y,z}[x,[y,z]_{\delta}]=	\circlearrowleft_{x,y,z}[\delta x,[y,z]],\ \ \ \forall x,y,z \in L.
	\end{equation}	
\end{proposition}
\begin{proof}
	Let $x,y,z \in L$. Then we have
	\begin{align*}
	\circlearrowleft_{x,y,z}[x,[y,z]_{\delta}]&=[x,\delta[y,z]]	+[y,\delta[z,x]]+[z,\delta[x,y]] \\
	&=[x,[\delta y,z]]+[x,[y,\delta z]]+[y,[\delta z,x]]+[y,[z,\delta x]]+[z,[\delta x,y]]+[z,[x,\delta y]].
	\end{align*}
	Using the Jacobi identity of $(L,[\cdot,\cdot])$, we get
	\begin{equation*}
	\circlearrowleft_{x,y,z}[x,[y,z]_{\delta}]=[\delta x,[y,z]]+[\delta y,[z,x]]+[\delta z,[x,y]].	
	\end{equation*}	
	This completes the proof.
\end{proof}
According to Theorem \ref{thm2.1} and Proposition \ref{prop2.2} we are able to announce the following defintion
\begin{definition}
	An \textbf{ InvDer-Lie algebra }  is a Lie algebra $(L,[\cdot,\cdot])$ equipped with an invertible derivation $\delta$ such that its inverse is a derivation satisfying the following condition
	\begin{align*}
	\circlearrowleft_{x,y,z}[\delta x,[y,z]]=0,
	\end{align*}
	which is called the InvDer-Jacobi identity, and is denoted by $(L,[\cdot,\cdot],\delta)$.
\end{definition}
\textbf{Notaion}: For the rest of the paper we write an \textbf{Inv-derivation} instead of (an invertible derivation  such that its inverse is a derivation) and the set of \textbf{Inv-derivations} of $L$ is denoted by \textbf{InvDer(L)}.\\
\begin{example}
	Let $L$ be a vector space, $\vartheta$ be a graded vector space and $s:\vartheta \rightarrow \vartheta$ be the suspension operator (see \cite{T1} for more details about the operator). We consider the graded vector space $DC^{\star}(L;L):=C^{\star}(L;L) \oplus sC^{\star}(L;L)=\displaystyle \oplus_{n\geq 0}(Hom(\wedge^{n+1}L,L)\times Hom(\wedge^{n}L,L))$. 
	Define a skew-symmetric bracket operation
	\begin{align*}
	\{\cdot,\cdot\}:DC^m(L;L)\otimes DC^n(L;L) \rightarrow DC^{m+n}(L;L),
	\end{align*}
	by 
	\begin{align*}
	\{(f_{m+1,g_m}),(f_{n+1,g_n})\}:=([f_{m+1},f_{n+1}]_{NR},(-1)^{m}[f_{m+1},g_n]_[NR]-(-1)^{n(m+1)}[f_{n+1},g_m]_{NR}),
	\end{align*}
	where $[\cdot,\cdot]_{NR}$ is the Nijenhuis-Richardson bracket. Then $L$ equipped with the Maurer-Cartan element $(w,\delta)$ of the graded Lie algebra $(DC^{\star}(L;L),\{\cdot,\cdot\})$ is an \textbf{InvDer-Lie algebra}, where $w \in Hom(\wedge^2L,L)$ and $\delta \in \textbf{InvDer(L)}$.
\end{example}
\begin{definition}
	Let $(L,\mu)$ be an algebra. A Rota-Baxter operator of weight $(\lambda)$ is a linear map $R:L \rightarrow L$ such that the following identity holds 
	\begin{equation}
	\mu(Rx,Ry)=R\Big(\mu(Rx,y)+\mu(x,Ry)+\lambda \mu(x,y) \Big), \ \ \forall x,y \in L.
	\end{equation}
\end{definition}
In the next of the paper, we consider only Rota-Baxter operator of weight $\lambda=0$. We need the previous operator in many topics like relating several algebraic structures, for example, the passage from Lie algebras to pre-Lie algebras, also from associative algebras to pre-Lie algebras like the following diagram:
\begin{center}
	$\xymatrix{
		\textbf{\emph{associative algebras}} \ar[r]^{\text{endomorphism operator}} \ar[dr]^{\text{R.B.O}}& \textbf{\emph{Lie algebras}} \ar[d]^{\text{R.B.O}}  \\
		&\textbf{\emph{pre-Lie algebras}}   &   
	}$
\end{center}
\begin{example}
	Let $(L,[\cdot,\cdot])$ be a Lie algebra and $R:L \rightarrow L$ be a Rota-Baxter operator of weight $\lambda=0$ such that $R$ is invertible and $R^{-1}$ is a Rota-Baxter operator of the same weight. Then $(L,[\cdot,\cdot],R)$ is an \textbf{InvDer-Lie algebra}.
\end{example}
\begin{example}
	Let $(L,[\cdot,\cdot])$ be an abelian Lie algebra and   $\alpha$ be an isomorphism of $L$. Then $(L,[\cdot,\cdot],\alpha)$ is an \textbf{InvDer-Lie algebra}.
\end{example}
Now, let's see the same result in the case of pre-Lie algebra
\begin{proposition}
	Let $(L,\ast)$ be a pre-Lie algebra and $\delta \in Der(L)$. Then
	\begin{equation} \label{eq2.6}
	\delta x \ast \delta y=\delta^2(x\ast y) \Leftrightarrow \delta \in InvDer(L).
	\end{equation}
\end{proposition}
The previous proposition leads us to announce the next theorem.
\begin{theorem}\label{thm2.2}
	Let $(L,\ast)$ be a pre-Lie algebra and $\delta \in \textbf{InvDer(L)}$. Then $(L,\ast_{\delta}=\delta \circ \ast)$ is a pre-Lie algebra, which is called the \textbf{ twisted pre-Lie algebra by an invertible derivation }.
\end{theorem}
\begin{proof}
	For every $x,y,z \in L$ we have 
	\begin{align*}
	x \ast_{\delta}(y \ast_{\delta}z)&=\delta (x \ast (y\ast_\delta z)) \\
	&=\delta x \ast ( y \ast_{\delta} z)+x \ast \delta^2(y \ast z) \\
	&=\delta x \ast (\delta y \ast z)+ \delta x \ast (y \ast \delta z)+x \ast \delta^2(y \ast z).
	\end{align*}
	Similarly, we get
	\begin{align*}
	(x \ast_{\delta} y)\ast_{\delta} z&=(\delta x \ast y)\ast \delta z+(x \ast \delta y)\ast \delta z+\delta^2(x\ast y)\ast z,\\
	y \ast_{\delta} (x \ast_{\delta} z)&=\delta y \ast (\delta x \ast z)+\delta y \ast (x \ast \delta z)+y \ast \delta^2(x\ast z) \\
	(y \ast_{\delta} x) \ast_{\delta}z&=(\delta y \ast x) \ast \delta z+(y \ast \delta x)\ast \delta z+\delta^2(y \ast x) \ast z.
	\end{align*}
	Since $(L,\ast)$ is a pre-Lie algebra, so 
	\begin{align*}
	x \ast_{\delta}(y \ast_{\delta}z)-(x \ast_{\delta} y)\ast_{\delta} z&=\delta x \ast (\delta y \ast z)+ \delta x \ast (y \ast \delta z)+x \ast \delta^2(y \ast z)\\
	&\ \ \ -(\delta x \ast y)\ast \delta z-(x \ast \delta y)\ast \delta z-\delta^2(x\ast y)\ast z\\
	&=(\delta x \ast (y \ast \delta z)-(\delta x \ast y)\ast \delta z)+(\delta x \ast (\delta y \ast z)\\
	&\ \ \ -(x \ast \delta y)\ast \delta z)+x \ast \delta^2(y \ast z)-\delta^2(x\ast y)\ast z \\
	&=y \ast (\delta x \ast \delta z)-(y \ast \delta x)\ast \delta z+(\delta x \ast \delta y) \ast  z+\delta y \ast (\delta x \ast  z)\\
	&\ \ \ -(\delta y \ast \delta x)\ast  z
	-x\ast (\delta y \ast \delta z)+\delta y \ast (x \ast \delta z)-(\delta y \ast x)\ast \delta z\\
	&\ \ \ +x \ast \delta^2(y \ast z)-\delta^2(x\ast y)\ast z.
	\end{align*}
	From equation \eqref{eq2.6} we have 
	\begin{align*}
	x \ast_{\delta}(y \ast_{\delta}z)-(x \ast_{\delta} y)\ast_{\delta} z&=y \ast \delta^2(x \ast z)-(y \ast \delta x) \ast \delta z+\delta^2(x\ast y) \ast z +\delta y \ast (\delta x \ast z)-\delta^2(y \ast x)\ast z \\
	&\ \ \ -x \ast \delta^2(y \ast z)+\delta y \ast (x \ast \delta z)-(\delta y \ast x) \ast \delta z + x \ast \delta^2(y \ast z)-\delta^2(x\ast y) \ast z \\
	&=y \ast_{\delta}(x \ast_{\delta} z)-(y \ast_{\delta} x)\ast_{\delta}z.
	\end{align*}
	This completes the proof.	
\end{proof}
\begin{proposition} \label{prop2.3}
	Let $(L,\ast)$ be a pre-Lie algebra and $\delta \in$ \textbf{InvDer(L)}. Then the following equation holds 
	\begin{equation}\label{eq2.7}
	\delta x\ast (y \ast z)-(x\ast y)\ast \delta z=\delta y \ast (x\ast z)-(y\ast x)\ast z, \ \ \forall x,y,z \in L.
	\end{equation}
	\begin{proof}
		Let $x,y,z \in L$. Then we have 
		$$x\ast_{\delta}(y \ast_{\delta}z)-(x\ast_{\delta}y)\ast_{\delta}z-y\ast_{\delta}(x\ast_{\delta}z)+(y\ast_{\delta}x)\ast_{\delta}z=0,$$
		and consequently 
		$$\delta \circ \Big(x\ast(y \ast_{\delta}z)-(x\ast_{\delta}y)\ast z-y\ast(x\ast_{\delta}z)+(y\ast_{\delta}x)\ast z \Big)=0.$$
		Since $\delta$ is an invertible derivation, the above equation implies 
		\begin{equation}\label{eq2.8}
		x\ast(y \ast_{\delta}z)-(x\ast_{\delta}y)\ast z-y\ast(x\ast_{\delta}z)+(y\ast_{\delta}x)\ast z=0,
		\end{equation}
		which leads to the following
		\begin{align*}
		0&=x\ast (\delta y\ast z)+x\ast (y\ast \delta z)-(\delta x\ast y)\ast z-(x\ast \delta y)\ast z-y\ast (\delta x\ast z)\\
		&\ \ \ -y\ast (x\ast \delta z)+(\delta y \ast x)\ast z+(y\ast \delta x)\ast z.
		\end{align*}
		Since $(L,\ast)$ is a pre-Lie algebra, we obtain
		\begin{align*}
		&\delta y \ast (x \ast z)-\delta x \ast (y\ast z)+(x\ast y)\ast \delta z-(y\ast x)\ast \delta z=0.
		\end{align*}
		This completes the proof.
	\end{proof}
	According to Theorem \ref{thm2.2} and Proposition \ref{prop2.3} we can announce the following definition.
	\begin{definition}
		An \textbf{ InvDer pre-Lie algebra }$(L,\ast,\delta)$ is a pre-Lie algebra $(L,\ast)$ equipped with $\delta \in$ \textbf{InvDer(L)} satisfying the following equation 
		\begin{equation*}
		\delta x \ast (y\ast z)-(x\ast y)\ast \delta z=\delta y \ast (x \ast z)-(y\ast x)\ast \delta z.	
		\end{equation*}
	\end{definition}
\end{proposition}
Let $(L,[\cdot,\cdot])$ and $(L,\ast)$ be a Lie algebra and a pre-Lie algebra, respectively. The relation between both of them (from Lie pre-Lie algebra to Lie algebra) is satisfied by the commutator 
\begin{equation}\label{eq2.9}
[x,y]_C=x \ast y -y \ast x.
\end{equation}
In the next proposition we will investigate this relation with the \textbf{ twisted Lie structure by an invertible derivation}.
\begin{proposition}
	Let $(L,\ast_{\delta})$ be a \textbf{ twisted pre-Lie algebra by an ivertible derivation}. Then a Lie algebra rises via the commutator \eqref{eq2.9} which its structure is given by 
	\begin{equation} \label{eq2.10}
	[x,y]_C=x\ast_{\delta} y-y \ast_{\delta} x, \ \ \forall x,y \in L.
	\end{equation} 
\end{proposition}
\begin{proof}
	Let $x,y,z \in L$. Then we have
	\begin{equation*}
	[x,y]_C=x\ast_{\delta} y-y \ast_{\delta} x=-(y \ast_{\delta} x-x\ast_{\delta} y)=-[y,x]_C.
	\end{equation*}
	So the skew-symmetry holds. Now, we prove the Jacobi identity as follows:
	\begin{align*}
	[x,[y,z]_C]_C&=x\ast_{\delta} [y,z]_C-[y,z]_C \ast_{\delta}  x \\
	&=x\ast_{\delta} (y \ast_{\delta} z-z \ast_{\delta} y)-(y \ast_{\delta} z-z \ast_{\delta} y) \ast_{\delta} x \\
	&=x\ast_{\delta} (y \ast_{\delta} z)-x\ast_{\delta}(z \ast_{\delta} y)-(y \ast_{\delta} z)\ast_{\delta}x+(z \ast_{\delta} y) \ast_{\delta} x.
	\end{align*}
	Similarly 
	\begin{align*}
	[y,[z,x]_C]_C&=y \ast_{\delta} (z \ast_{\delta}x)-y\ast_{\delta}(x \ast_{\delta}z)-(z\ast_{\delta}x)\ast_{\delta}y+(x\ast_{\delta}z)\ast_{\delta}y, \\
	[z,[x,y]_C]_C&=z \ast_{\delta}(x\ast_{\delta}y)-z\ast_{\delta}(y\ast_{\delta}x)-(x\ast_{\delta}y)\ast_{\delta}z+(y\ast_{\delta}x)\ast_{\delta}z.
	\end{align*}
Since $(L,\ast_{\delta})$ is a \textbf{ twisted pre-Lie algebra by an invertible derivation}, we obtain 
	\begin{equation}\label{eq2.11}
	\circlearrowleft_{x,yz}[x,[y,z]_C]_C=0.
	\end{equation}
	This completes the proof.
\end{proof}
\section{InvDer Lie and InvDer pre-Lie algebras and Rota-Baxter operators}
In the next proposition, we will investigate the passage from \textbf{InvDer pre-Lie algebra} to a \textbf{InvDer Lie algebra} via the commutator \eqref{eq2.9}.
\begin{proposition}
	Let $(L,\ast,\delta)$ be an \textbf{InvDer pre-Lie algebra}. Then $(L,[\cdot,\cdot]_C,\delta)$ is an \textbf{InvDer Lie algebra}, where 
	\begin{equation*}
	[x,y]_C=x \ast y-y\ast x, \text{ for all }x,y \in L.
	\end{equation*}
\end{proposition}
\begin{proof}
	Let $x,y,z \in L$. It is clear that $[\cdot,\cdot]_C$ is a skew symmetric operation. Let's focus on proving the InvDer-Jacobi identity as follows: 
	\begin{align*}
	[\delta x,[y,z]_C]_C&=\delta x \ast [y,z]_C-[y,z]_C\ast \delta x \\
	&=\delta x\ast(y\ast z-z\ast y)-(y\ast z-z\ast y)\ast \delta x \\
	&=\delta x\ast(y\ast z)-\delta x\ast (z\ast y)-(y\ast z)\ast \delta x+(z\ast y)\ast \delta x.
	\end{align*}
	Similarly, we obtain 
	\begin{align*}
	[\delta y,[z,x]_C]_C&=\delta y\ast(z\ast x)-\delta y\ast (x\ast z)-(z\ast x)\ast \delta y+(x\ast z)\ast \delta y,\\
	[\delta z,[x,y]_C]_C&=\delta z\ast(x\ast y)-\delta z\ast (y\ast x)-(x\ast y)\ast \delta z+(y\ast x)\ast \delta z.
	\end{align*}
	Since $(L,\ast,\delta)$ is an \textbf{InvDer pre-Lie algebra}, we obtain  
	\begin{align*}
	\circlearrowleft_{x,y,z}[\delta x,[y,z]_C]_C&=\delta x\ast(y\ast z)-\delta x\ast (z\ast y)-(y\ast z)\ast \delta x+(z\ast y)\ast \delta x \\
	&+\delta y\ast(z\ast x)-\delta y\ast (x\ast z)-(z\ast x)\ast \delta y+(x\ast z)\ast \delta y \\
	&+\delta z\ast(x\ast y)-\delta z\ast (y\ast x)-(x\ast y)\ast \delta z+(y\ast x)\ast \delta z \\
	&=0.
	\end{align*}
	So now, we need just to prove that $\delta$ and $\delta^{-1}$ are both derivations on $(L,[\cdot,\cdot])$. We have
	\begin{align*}
	\delta ([x,y]_C)&=\delta (x\ast y-y\ast x) \\
	&=\delta x \ast y+x \ast \delta y-\delta y\ast x-y\ast \delta x \\
	&=\delta x \ast y-y\ast \delta x+x \ast \delta y-\delta y\ast x \\
	&=[\delta x,y]_C+[x,\delta y]_C,
	\end{align*}
	i.e., $\delta$ is a derivation. In the similar way, we can see that $\delta^{-1}$ is a derivation. This completes the proof.
\end{proof}
In the next proposition we consider a Rota-Baxter of weight ($\lambda=0$) and use it to verify the passage from \textbf{InvDer Lie algebra} to \textbf{InvDer pre-Lie algebra}.
\begin{proposition}
	Let $(L,[\cdot,\cdot],\delta)$ be an \textbf{InvDer Lie algebra}, $R$ is a Rota-Baxter operator and suppose that \begin{equation}\label{eq3.1}
	\delta \circ R=R\circ \delta \text{ and } \delta^{-1} \circ R=R\circ \delta^{-1}.
	\end{equation} 
	Then $(L,\ast,\delta)$ is an \textbf{InvDer pre-Lie algebra}, where 
	\begin{equation}\label{eq3.2}
	x\ast y=[Rx,y], \ \ \ \forall x,y \in L.
	\end{equation}
\end{proposition}
\begin{proof}
	Let $x,y,z \in L$. Using \eqref{eq3.1} we have
	\begin{align*}
	\delta x \ast (y\ast z)=[R\circ \delta x,y\ast z] =[R\circ \delta,[Ry,z]]=[\delta \circ Rx,[Ry,z]].
	\end{align*}
	Using the InvDer-Jacobi identity, we obtain 
	\begin{equation*}
	\delta x \ast (y\ast z)=-[\delta \circ Ry,[z,Rx]]+[[Rx,Ry],\delta z].
	\end{equation*}
	On the other side, we have 
	\begin{align*}
	(x\ast y)\ast \delta z&=[R(x\ast y),\delta z]=[R[Rx,y],\delta z].
	\end{align*}
	Now, we get
	\begin{align*}
	\delta x \ast (y\ast z)-(x\ast y)\ast \delta z&=-[\delta \circ Ry,[z,Rx]]+[[Rx,Ry],\delta z]-[R[Rx,y],\delta z].
	\end{align*}
	Using the identity of Rota-Baxter operator and \eqref{eq3.1} in the above equation we deduce that
	\begin{align*}
	\delta x \ast (y\ast z)-(x\ast y)\ast \delta z&=-[\delta \circ Ry,[z,Rx]]+[R[x,Ry],\delta z] \\
	&=[\delta \circ Ry,[Rx,z]]-[R[Ry,x],\delta z] \\
	&=[R\circ \delta y,[Rx,z]]-[R[Ry,x],\delta z] \\ 
	&=\delta y\ast(x\ast z)-(y\ast x)\ast \delta z.
	\end{align*}
	Now, we prove that $\delta \text{ and } \delta^{-1}$ are both derivations on $(L,\ast)$. Using \eqref{eq3.1} we get
	\begin{align*}
	\delta (x\ast y)&=\delta([Rx,y])=[\delta \circ Rx,y]+[Rx,\delta y]=[R\circ \delta x,y]+[Rx,\delta y] =\delta x\ast y+x\ast \delta y.
	\end{align*}
	So $\delta$ is a derivation. In the similar way, we deduce that $\delta^{-1}$ is a derivation. This completes the proof.
\end{proof}
In the next of this section, we construct an \textbf{ InvDer associative algebra} by twisting an associative algebra by an invertible derivation. Let's first start with the following proposition in which we study the impact of $\delta \in \textbf{InvDer(L)}$.
\begin{proposition}
	Let $(L,\mu)$ be an associative algebra and $\delta \in Der(L)$. Then 
	\begin{equation} \label{eq3.3}
	\delta^2(\mu(x,y))=\mu(\delta x,\delta y)\Leftrightarrow \delta \in \textbf{ InvDer(L)}, \text{ for all } x,y \in L.
	\end{equation}
\end{proposition}
\begin{theorem} \label{thm3.4}
	Let $(L,\mu)$ be an associative algebra and $\delta \in \textbf{ InvDer(L)}$. Then $(L,\mu_{\delta}=\delta \circ \mu)$ is an associative algebra which is denoted by \textbf{ twisted associative algebra by an invertible derivation}.
\end{theorem}
\begin{proof}
	Let $x,y,z \in L$. Using \eqref{eq3.3} we get
	\begin{align*}
	\mu_{\delta}(x,\mu_{\delta}(y,z))&=\delta \circ \mu(x,\mu_{\delta}(y,z)) \\
	&=\mu(\delta x,\mu_{\delta}(y,z))+\mu(x,\delta \circ \mu_{\delta}(y,z)) \\
	&=\mu(\delta x,\mu(\delta y,z))+\mu(\delta x,\mu(y,\delta z))+\mu(x,\delta^2\mu(y,z)) \\
	&=\mu(\delta x,\mu(\delta y,z))\mu(\delta x,\mu(y,\delta z))+\mu(x,\mu(\delta y,\delta z)) \\
	&=\mu(\mu(\delta x,\delta y),z)+\mu(\mu(\delta x,y),\delta z)+\mu(\mu(x,\delta y),\delta z) \\
	&=\mu_{\delta}(\mu_{\delta}(x,y),z).
	\end{align*}
	This completes the proof.
\end{proof}
\begin{proposition}
	Let $(L,\mu)$ be an associative algebra and $\delta \in \textbf{ InvDer(L)}$. Then the following identity holds 
	\begin{equation} \label{eq3.4}
	\mu(\delta x,\mu(y,z))=\mu(\mu(x,y),\delta z).
	\end{equation}
\end{proposition}
\begin{proof}
	From Theorem \eqref{thm3.4} we have
	\begin{equation*}
	\mu_{\delta}(x,\mu_{\delta}(y,z))=\mu_{\delta}(\mu_{\delta}(x,y),z), \ \ \ \forall x, y, z\in L. 	
	\end{equation*}
	So, we get
	\begin{align*}
	\delta\Big(\mu(x,\mu_{\delta}(y,z))-\mu(\mu_{\delta}(x,y),z) \Big)=0.
	\end{align*}
	Since $\delta \in \textbf{InvDer(L)}$, then $Ker(\delta)=\{0\}$ and consequently 
	$$ \mu(x,\mu_{\delta}(y,z))-\mu(\mu_{\delta}(x,y),z)=0.$$
	The above equation gives us
	$$\mu(x,\mu(y,\delta z))-\mu(\mu(\delta x,y),z)=0.$$
	Since $(L,\mu)$ is an associative algebra, we deduce that
	$\mu(\delta x,\mu(y,z))=\mu(\mu(x,y),\delta z)$, which completes the proof.
\end{proof}
With the previous results we can define \textbf{ InvDer associative algebras} in the following:
\begin{definition}
	An \textbf{ InvDer associative algebras} $(L,\mu,\delta)$ is an associative algebras $(L,\mu)$ equipped with $\delta \in \textbf{ InvDer(L)}$ satisfying the following equation 
	\begin{equation*}
	\mu(\delta x,\mu(y,z))=\mu(\mu(x,y),\delta z), \ \ \forall  x,y,z \in L.
	\end{equation*}
\end{definition}
Now, let's investigate the link between \textbf{ InvDer Lie and associative algebras} in the next proposition.
\begin{proposition}
	Let $(L,\mu,\delta)$ be an \textbf{ InvDer associative algebra }. Then $(L,[\cdot,\cdot]_C,\delta)$ is an \textbf{ InvDer Lie algebra}, where 
	\begin{equation} \label{eq3.5}
	[x,y]_C=\mu(x,y)-\mu(y,x), \ \ \forall x,y \in L.
	\end{equation}  
\end{proposition}
\begin{proof}
	First we check the skew symmetry property. We have 
	\begin{equation*}
	[x,y]_C=\mu(x,y)-\mu(y,x)=-(\mu(y,x)-\mu(x,y))=-[y,x]_C,
	\end{equation*}
	where $x, y, z\in L$. Now, we check the InvDer Jacobi identity. Easily we get
	\begin{align*}
	[\delta x,[y,z]_C]_C&=\mu(\delta x,[y,z]_C)-\mu([y,z]_C,\delta x) \\
	&=\mu(\delta x,\mu(y,z))-\mu(\delta x,\mu(z,y))-\mu(\mu(y,z),\delta x)+\mu(\mu(z,y),\delta x).
	\end{align*}
	Similarly, we obtain 
	\begin{align*}
	[\delta y,[z,x]_C]_C&=\mu(\delta y,\mu(z,x))-\mu(\delta y,\mu(x,z))-\mu(\mu(z,x),\delta y)+\mu(\mu(x,z),\delta y),\\
	[\delta z,[x,y]_C]_C&=\mu(\delta z,\mu(x,y))-\mu(\delta z,\mu(y,x))-\mu(\mu(x,y),\delta z)+\mu(\mu(y,x),\delta z).
	\end{align*}
	Since $(L,\mu,\delta)$ is an \textbf{ InvDer associative algebra }, the simple calculations imply
	$\circlearrowleft_{x,y,z}[\delta x,[y,z]_C]_C=0$.
	It is easy to see that  $\delta$ and $\delta^{-1}$ are two derivations on $(L,[\cdot,\cdot]_C)$. This completes the proof.
\end{proof}
In the next, we construct an \textbf{ InvDer Lie and pre-Lie algebras } from \textbf{ InvDer associative algebras} via some operators like endomorphism and Rota-Baxter operator. In the following proposition we construct an \textbf{ InvDer Lie algebras } from \textbf{ InvDer associative algebras} via an endomorphism operator with an additional conditions.
\begin{proposition}
	Let $(L,\mu,\delta)$ be an \textbf{ InvDer associative algebras}, $R : L \rightarrow L$ be an endomorphism operator satisfying $R^2x=Rx$ for $x\in L$, $\delta \circ R= R \circ \delta $ and $\delta ^{-1} \circ R=R \circ \delta ^{-1}$. Then we can define an \textbf{ InvDer Lie algebras } structure on $L$ by 
	\begin{equation} \label{eq3.6}
	[x,y]=\mu(Rx,y)-\mu(Ry,x), \ \ \forall x,y \in L.
	\end{equation}
\end{proposition}
\begin{proof}
	The skew symmetry property follows from the following:
	\begin{equation*}
	[x,y]=\mu(Rx,y)-\mu(Ry,x)=-(\mu(Ry,x)-\mu(Rx,y))=-[y,x],\ \ \forall x, y\in L.
	\end{equation*}
	Now, we check the InvDer Jacobi identity. From the properties of $R$ we get 
	\begin{align*}
	[\delta x,[y,z]]&=\mu(R\circ \delta x,[y,z])-\mu(R([y,z]),\delta x) \\
	&=\mu(R\circ \delta x,\mu(Ry,z))-\mu(R([y,z]),\delta x)-\mu(\mu(R^2y,Rz),\delta x)+\mu(\mu(R^2z,Ry),\delta x) \\
	&=\mu(\delta \circ Rx,\mu(Ry,z))-\mu(\delta \circ Rx,\mu(Rz,y))-\mu(\mu(Ry,Rz),\delta x)+\mu(\mu(Rz,Ry),\delta x),
	\end{align*}
	where $x, y, z\in L$. Similarly, we obtain 
	\begin{align*}
	[\delta y,[z,x]]&=\mu(\delta \circ Ry,\mu(Rz,x))-\mu(\delta \circ Ry,\mu(Rx,z))-\mu(\mu(Rz,Rx),\delta y)+\mu(\mu(Rx,Rz),\delta y), \\
	[\delta z,[x,y]]&=\mu(\delta \circ Rz,\mu(Rx,y))-\mu(\delta \circ Rz,\mu(Ry,x))-\mu(\mu(Rx,Ry),\delta z)+\mu(\mu(Ry,Rx),\delta z). 
	\end{align*}
	Since $(L,\mu,\delta)$ is an \textbf{ InvDer associative algebra}, we obtain 
	\begin{equation*}
	\circlearrowleft_{x,y,z}[\delta x,[y,z]]=0.
	\end{equation*}
	The last step in the proof is to show that $\delta$ and $\delta^{-1}$ are two derivations on $(L,[\cdot,\cdot])$. Direct calculations and the property $\delta \circ R=R\circ \delta$ imply
	\begin{align*}
	\delta([x,y])&=\delta(\mu(Rx,y)-\mu(Ry,x)) \\
	&=\mu(\delta \circ Rx,y)+\mu(Rx,\delta y)-\mu(\delta \circ Ry,x)-\mu(Ry,\delta x) \\
	&=\mu(R\circ \delta x,y)+\mu(Rx,\delta y)-\mu(R\circ \delta y,x)-\mu(Ry,\delta x) \\
	&=[\delta x,y]+[x,\delta y], 
	\end{align*}
	i.e., $\delta$ is a derivation on $(L,[\cdot,\cdot])$. The same result holds for $\delta^{-1}$. This completes the proof.
\end{proof}
Now we investigate, in the next theorem, the same relation by using a \textbf{ Rota-Baxter operator } of weight $(\lambda=0)$.
\begin{theorem}
	Let $(L,\mu,\delta)$ be an \textbf{ InvDer associative algebra}. Then $(L,\ast,\delta)$ is an \textbf{ InvDer pre-Lie algebra}, where
	\begin{equation}
	x\ast y=\mu(Rx,y)-\mu(y,Rx), \ \forall x,y \in L,
	\end{equation}
	such that $\delta \circ R=R \circ \delta$ and $\delta^{-1} \circ R=R \circ \delta^{-1}$.
\end{theorem}
\begin{proof}
	For any $x,y,z\in L$, we get 
	\begin{align*}
	\delta x \ast (y\ast z)&=\mu(R\circ \delta x,\mu(Ry,z))-\mu(R\circ \delta x,\mu(z,Ry))-\mu(\mu(Ry,z),R\circ \delta x)+\mu(\mu(z,Ry),R\circ \delta x),\\
	(x\ast y)\ast \delta z&=\mu(R \circ \mu(Rx,y),\delta z)-\mu(R\circ\mu(y,Rx),\delta z)-\mu(\delta z,R\circ\mu(Rx,y))+\mu(\delta z,R\circ\mu(y,Rx)).
	\end{align*}
	So, using the property $\delta \circ R=R\circ \delta$ and the associativity of $(L,\mu,\delta)$, we obtain
	\begin{align*}
	\delta x\ast(y\ast z)-(x\ast y)\ast \delta z&=\mu(\delta \circ Rx,\mu(Ry,z))-\mu(\delta \circ Rx,\mu(z,Ry))-\mu(\mu(Ry,z),\delta \circ Rx) \\
	&\ \ +\mu(\mu(z,Ry),\delta \circ Rx)-\mu(R\circ \mu(Rx,y),\delta z)+\mu(R\circ \mu(y,Rx),\delta z)\\
	&\ \ +\mu(\delta z,R\circ \mu(Rx,y))-\mu(\delta z,R\circ \mu(y,Rx)) \\
	&=\mu(\mu(Rx,Ry),\delta z)-\mu(\delta \circ Rx,\mu(z,Ry))-\mu(\mu(Ry,z),\delta \circ Rx) \\
	&\ \ +\mu(\delta z,\mu(Ry,Rx))-\mu(R\circ \mu(Rx,y),\delta z)+\mu(R\circ \mu(y,Rx),\delta z)\\
	&\ \ +\mu(\delta z,R\circ \mu(Rx,y))-\mu(\delta z,R\circ \mu(y,Rx)).
	\end{align*}
	Using the Rota-Baxter operator, the above equation gives us:
    \begin{align*}
	\delta x\ast(y\ast z)-(x\ast y)\ast \delta z&=\mu(R\circ \mu(Rx,y),\delta z)+\mu(R\circ \mu(x,Ry),\delta z)-\mu(\delta \circ Rx,\mu(z,Ry)) \\
	&\ \ -\mu(\mu(Ry,z),\delta \circ Rx)+\mu(\delta z,R\circ \mu(Ry,x))+\mu(\delta z,R\circ \mu(y,Rx))\\
	&\ \ -\mu(R\circ \mu(Rx,y),\delta z)+\mu(R\circ \mu(y,Rx),\delta z)+\mu(\delta z,R\circ \mu(Rx,y))\\
	&\ \ -\mu(\delta z,R\circ \mu(y,Rx))\\
	&=\mu(R\circ \mu(x,Ry),\delta z)-\mu(\delta \circ Rx,\mu(z,Ry))-\mu(\mu(Ry,z),\delta \circ Rx) 
	\\
	&\ \ +\mu(\delta z,R\circ \mu(Ry,x))+\mu(R\circ \mu(y,Rx),\delta z)
	+\mu(\delta z,R\circ \mu(Rx,y)).
	\end{align*}
	By reusing the Rota-Baxter operator in the above equation, we get
	\begin{align*}
		\delta x\ast(y\ast z)-(x\ast y)\ast \delta z&=\mu(R\circ \mu(x,Ry),\delta z)-\mu(\delta \circ Rx,\mu(z,Ry))-\mu(\mu(Ry,z),\delta \circ Rx) 
	\\
	&\ \ +\mu(\delta z,R\circ \mu(Ry,x))+\mu(\mu(Ry,Rx),\delta z)-\mu(R\circ \mu(Ry,x),\delta z)\\
	&\ \ +\mu(\delta z,\mu(Rx,Ry))-\mu(\delta z,R\circ \mu(x,Ry)) \\
	&=y\ast(x\ast z)-(y\ast x)\ast z.
	\end{align*}	
\end{proof}
\section{InvDer zinbiel and dendriform algebras}
In this section, we study the impact of an invertible derivation ($\delta \in InvDer(L)$) on zinbiel and dendriform algebras and define \textbf{InvDer zinbiel algebras} and \textbf{InvDer dendriform algebras}. \\
\subsection{InvDer zinbiel algebras}
The category of zinbiel algebras is Koszul dual to the category of Leibniz algebras in the sens of J. L. Loday (see \cite{L1} for more details). A zinbiel algebra is an algebra $(L,\diamond)$ consisting of a vector space $L$ and a bilinear map $\diamond:L\otimes L \rightarrow L$ such that
\begin{equation*}
x \diamond (y \diamond z)=(x \diamond y) \diamond z+(y\diamond x) \diamond z, \ \ \forall x,y,z \in L.
\end{equation*} 
A linear map $\delta:L \rightarrow L$ is said to be a derivation on a zinbiel algebra $(L,\diamond)$ if it satisfies the following equation
\begin{equation*}
\delta (x \diamond y)=(\delta x \diamond y)+(x \diamond \delta y), \text{ for } x,y \in L.
\end{equation*}
We denote by $Der(L)$ the set of all derivation.
\begin{proposition}
	Let $(L,\diamond)$ be a zinbiel algebra and $\delta$ be an invertible derivation. Then, for $x,y \in L$
	\begin{equation} \label{eq4.1}
	\delta^2(x \diamond y)=(\delta x \diamond \delta y) \Leftrightarrow \delta \in InvDer(L).
	\end{equation}
\end{proposition}
The previous proposition is important to construct from a zinbiel algebra a new one which is called the \textbf{twisted zinbiel algebra by invertible derivation}.
\begin{theorem}\label{Th4.2}
	Let $(L,\diamond)$ be a zinbiel algebra and $\delta \in InvDer(L)$. Then $(L,\diamond_{\delta}=\delta \circ \diamond)$ is a zinbiel algebra, which is called the \textbf{twisted zinbiel algebra by invertible derivation}.
\end{theorem}
\begin{proof}
	For any $x,y,z \in L$, we get
	\begin{align*}
	x \diamond_{\delta} (y \diamond_{\delta} z)&=\delta(x \diamond (y \diamond_{\delta} z)) \\
	&=\delta x \diamond (y \diamond_{\delta}z)+x \diamond \delta^2(y\diamond z) \\
	&\overset{\eqref{eq4.1}}{=}\delta x \diamond (\delta y \diamond z)+\delta x \diamond (y\diamond \delta z)+x\diamond (\delta y \diamond \delta z) \\
	&=(\delta x \diamond \delta y)\diamond z+(\delta y \diamond \delta x) \diamond z+(\delta x \diamond y)\diamond \delta z+(y \diamond \delta x)\diamond \delta z+(x \diamond \delta y)\diamond \delta z+(\delta y\diamond x)\diamond \delta z \\
	&=((\delta x \diamond \delta y)\diamond z+(\delta x \diamond y)\diamond \delta z+(x \diamond \delta y)\diamond \delta z)+ ((\delta y \diamond \delta x) \diamond z+(y \diamond \delta x)\diamond \delta z+(\delta y\diamond x)\diamond \delta z ) \\
	&=(x \diamond_{\delta}y)\diamond_{\delta}z+(y\diamond_{\delta}x)\diamond_{\delta}z,
	\end{align*}
	which completes the proof.
\end{proof}
\begin{proposition}\label{Pr4.3}
	Let $(L,\diamond_{\delta})$ be a \textbf{twisted zinbiel algebra by invertible derivation}. Then the following equation holds 
	\begin{equation}\label{eq4.2}
	\delta x \diamond (y \diamond z)=(x\diamond y)\diamond \delta z+(y\diamond x)\diamond \delta z,\ \ \forall x,y,z \in L.
	\end{equation}
\end{proposition}
\begin{proof}
	For any  $x,y,z \in L$, we have
	\begin{align*}
	x \diamond_{\delta} (y\diamond_{\delta}z)-(x\diamond_{\delta}y)\diamond_{\delta}z-(y\diamond_{\delta}x)\diamond_{\delta}z=0,
	\end{align*}
	and consequently 
	\begin{align*}
	\delta \Big(x \diamond (y\diamond_{\delta}z)-(x\diamond_{\delta}y)\diamond z-(y\diamond_{\delta}x)\diamond z \Big)=0.
	\end{align*}
	Since $\delta \in InvDer(L)$, the above equation implies
	\begin{align*}
	0&=x \diamond (y\diamond_{\delta}z)-(x\diamond_{\delta}y)\diamond z-(y\diamond_{\delta}x)\diamond z \\
	&=x\diamond(\delta y\diamond z)+x\diamond(y \diamond \delta z)-(\delta x\diamond y)\diamond z-(x\diamond \delta y)\diamond z-(\delta y\diamond x)\diamond z-(y\diamond \delta x)\diamond z.
	\end{align*}
	Since $(L,\diamond)$ is a zinbiel algebra, the above equation reduces to 
	\begin{align*}
	x\diamond(y\diamond \delta z)-(\delta x\diamond y)\diamond z-(\delta y \diamond x)\diamond z=0,
	\end{align*}
	or 
	\begin{align*}
	(x\diamond y)\diamond \delta z+(y\diamond x)\diamond \delta z-\delta x\diamond(y\diamond z)=0.
	\end{align*}
	This completes the proof.
\end{proof}
According to Theorem \ref{Th4.2} and Proposition \ref{Pr4.3}, we can give the following definition
\begin{definition}
	An \textbf{InvDer zinbiel algebra} is a triple $(L,\diamond,\delta)$ consisting of a vector space $L$, a bilinear map $\diamond :L\otimes L \rightarrow L$ and a derivation $\delta$ such that $\delta \in InvDer(L)$ satisfying the following identity: 
	\begin{equation} \label{eq4.3}
	\delta x \diamond (y\diamond z)=(x\diamond y)\diamond \delta z+(y \diamond x)\diamond \delta z, \ \ \forall x,y,z \in L.
	\end{equation}
\end{definition}
\begin{proposition}
	Let $(L,\diamond,\delta)$ be an \textbf{InvDer zinbiel algebra}. Then the following equations holds 
	\begin{equation} \label{eq4.4}
	\delta x \diamond (z \diamond y)=\delta z \diamond (x \diamond y),
	\end{equation}
	\begin{equation}\label{eq4.5}
	(x\diamond y)\diamond \delta z=(x\diamond z)\diamond \delta y,
	\end{equation}
	where $x,y,z \in L$.
\end{proposition}
\begin{proof}
	For any $x,y,z \in L$, we get
	\begin{align*}
	\delta x \diamond (z \diamond y)=(x\diamond z)\diamond \delta y+(z\diamond x) \diamond \delta y=(z\diamond x) \diamond \delta y+(x\diamond z)\diamond \delta y=\delta z \diamond (x \diamond y),
	\end{align*}
	and 
	\begin{align*}
	(x\diamond y)\diamond \delta z=\delta x\diamond(y\diamond z)+\delta x \diamond (z \diamond y)=\delta x \diamond (z \diamond y)+\delta x\diamond(y\diamond z)=(x\diamond z) \diamond \delta y.
	\end{align*}
\end{proof}
\begin{proposition}
	Let $(L,\diamond,\delta)$ be an \textbf{InvDer zinbiel algebra}. Then $(L,\mu)$ is a commutative \textbf{InvDer associative algebra}, where the bilinear map $\mu:L \rightarrow L$ depends on $\diamond$ as mentioned bellow
	\begin{equation*}
	\mu(x,y)=x\diamond y+y \diamond x,\ \ \forall x,y \in L.
	\end{equation*}
\end{proposition}
\begin{proof}
	Let $x,y,z \in L$. Since $(L,\diamond,\delta)$ is an \textbf{InvDer zinbiel algebra}, so using \eqref{eq4.4} the following calculus holds
	\begin{align*}
	\mu(\delta x,\mu(y,z))&=\delta x \diamond \mu(y,z)+\mu(y,z) \diamond \delta x \\
	&=\delta x \diamond (y\diamond z)+\delta x\diamond (z\diamond y)+(y\diamond z)\diamond \delta x+(z\diamond y)\diamond \delta x \\
	&=(x\diamond y)\diamond \delta z+(y\diamond x)\diamond \delta z+\delta z \diamond (x \diamond y)-(z\diamond y)\diamond \delta x+\delta z\diamond (y\diamond x)+(z\diamond y)\diamond \delta x \\
	&=(x\diamond y)\diamond \delta z+(y\diamond x)\diamond \delta z+\delta z \diamond (x \diamond y)+\delta z\diamond (y\diamond x) \\
	&=\mu(\mu(x,y),\delta z).
	\end{align*}
	This completes the proof, because the commutativity is trivial.
\end{proof}
\begin{theorem}
	Let $(L,\diamond,\delta)$ be an \textbf{InvDer zinbiel algebra}. Then $(L,[\cdot,\cdot],\delta)$ is an \textbf{InvDer Lie algebra}, where the Lie bracket $[\cdot,\cdot]$ depends on the bilinear map $\diamond$ as mentioned in the following equation: 
	\begin{equation}
	[x,y]=x\diamond y-y\diamond x,\ \ \forall x,y \in L.
	\end{equation}
\end{theorem}
\begin{proof}
	First, we investigate the skew symmetry property as follows:  
	\begin{equation*}
	[x,y]=x\diamond y-y\diamond x=-(y\diamond x-x\diamond y)=-[y,x], \ \ \forall x,y,z \in L.
	\end{equation*}
	Now, we focus on the \textbf{InvDer Jacobi} property. We have
	\begin{align*}
	[\delta x,[y,z]]&=\delta x \diamond [y,z]-[y,z] \diamond \delta x \\
	&=\delta x\diamond (y\diamond z-z\diamond y)-(y\diamond z-z\diamond y) \diamond \delta x \\
	&=\delta x\diamond (y\diamond z)-\delta x \diamond (z\diamond y)-(y\diamond z)\diamond\delta x+(z\diamond y)\diamond \delta x.
	\end{align*}
	Similarly, we obtain 
	\begin{align*}
	[\delta y,[z,x]]&=\delta y\diamond(z\diamond x)-\delta y\diamond(x\diamond z)-(z\diamond x)\diamond \delta y+(x\diamond z)\diamond \delta y, \\
	[\delta z,[x,y]]&=\delta z\diamond (x\diamond y)-\delta z \diamond(y\diamond x)-(x\diamond y)\diamond\delta z+(y\diamond x)\diamond\diamond z.
	\end{align*}
	Using \eqref{eq4.4} and \eqref{eq4.5} and considering  that $(L,\diamond,\delta)$ is an \textbf{InvDer zinbiel algebra}, we obtain 
	\begin{equation*}
	\circlearrowleft_{x,y,z}[\delta x,[y,z]]=0,
	\end{equation*}
	which completes the proof.
\end{proof}
\subsection{InvDer dendriform algebras}
A dendriform algebra is a vector space $L$ equipped with two binary operations $\prec$ and $\succ$ which they satisfy three new identities. Loday had defined the notion of free dendriform algebras also its cohomology in \cite{L2}.

Recall that if $(L,\prec,\succ)$ is a dendriform algebra and $\delta:L \rightarrow L$ a linear map, then $\delta$ is said to be a derivation on $(L,\prec,\succ)$ if it satisfies the following conditions:
\begin{equation*}
\delta(x\prec y)=\delta x\prec y+x\prec\delta y,\ \ \ \delta(x\succ y)=\delta x\succ y+x\succ\delta y,\ \ \forall x,y \in L.
\end{equation*}
We denote by $Der(L)$ the set of all derivations on $L$.
\begin{proposition}
	Let $(L,\prec,\succ)$ be a dendriform algebra and $\delta$ be an invertible derivation. Then, for any $x,y \in L$, we have
	\begin{align*}
	\delta^2(x\prec y)=\delta x \prec \delta y \Leftrightarrow \delta \in \textbf{InvDer(L)},\ \ \ \delta^2(x\succ y)=\delta x \succ \delta y \Leftrightarrow \delta \in \textbf{InvDer(L)}.
	\end{align*}
\end{proposition}
\begin{theorem}
	Let $(L,\prec,\succ)$ be a dendriform algebra and $\delta \in \textbf{InvDer(L)}$. Then $(L,\prec_{\delta}=\delta \circ \prec,\succ_{\delta}=\delta \circ \succ)$ is also a dendriform algebra which is called \textbf{twisted dendriform algebra by invertible derivation}.
\end{theorem}
\begin{proposition}
	Let ($L,\prec_{\delta},\succ_{\delta}$) be a \textbf{twisted dendriform algebra by invertible derivation}. Then, we have 
	\begin{eqnarray}
	(x\prec y)\prec \delta z&=&\delta x \prec(y\prec z+y\succ z),\label{eq4.7} \\ 
	(x\succ y)\prec \delta z&=&\delta x\succ (y\prec z),\label{eq4.8} \\
	\delta x\succ(y\succ z)&=&(x\prec y+x\succ y) \succ \delta z,\label{eq4.9}
	\end{eqnarray}
	where $x,y \in L$.
\end{proposition}
\begin{proof}
	Since $\delta \in \textbf{InvDer(L)}$ and ($(L,\prec_{\delta},\succ_{\delta}$) is a twisted dendrifrom algebra by invertible derivation, we get
	\begin{align*}
	(x\prec_{\delta}y)\prec_{\delta} z-x\prec_{\delta}(y\prec_{\delta}z+y\succ_{\delta}z)=0,
	\end{align*}
	which gives us
	\begin{align*}
	\delta \circ (\delta (x\prec y)\prec z-x\prec\delta(y\prec z)-x\prec \delta(y\succ z)).
	\end{align*}
	Since $\delta$ is invertible, the above equation implies
	\begin{align*}
	0&=\delta (x\prec y)\prec z-x\prec\delta(y\prec z)-x\prec \delta(y\succ z).
	\end{align*}
	From the above equation, we get
    \begin{align*}
	(\delta x\prec y)\prec z-x\prec(y\prec \delta z)-x\prec(y\succ \delta z)=0,
	\end{align*}
	and consequently
    \begin{align*}
	\delta x\prec(y\prec z+y\succ z)-(x\prec y)\prec \delta z=0.
	\end{align*}
	So the first assertion holds. Now, we check the second one. Using the definition of $\prec_{\delta}$ and $\succ_{\delta}$ we get
	\begin{align*}
	0&=(x\succ_{\delta} y)\prec_{\delta} z-x\succ_{\delta}(y\prec_{\delta}z)=\delta(x\succ y)\prec z-x\succ \delta (y\prec z)\\
	&=(\delta x\succ y)\prec z+(x\succ\delta y)\prec z-x\succ(\delta y\prec z)-x\succ(y\prec\delta z) \\
	&=(\delta x\succ y)\prec z-x\succ(y\prec\delta z)\\
	&=\delta x\succ(y\prec z)-(x\succ y)\prec\delta z.
	\end{align*}
	The last one is similar to the first assertion. This completes the proof.
\end{proof}
Now, the previous results allow us to announce the following definition
\begin{definition}
	An \textbf{InvDer dendriform algebra} is a quadruple $(L,\prec,\succ,\delta)$, where $(L,\prec,\succ)$ is a dendriform algebra and $\delta \in \textbf{InvDer(L)}$ such that, for all $x,y,z\in L$,
	\begin{align*}
	(x\prec y)\prec \delta z&=\delta x \prec(y\prec z+y\succ z), \\ 
	(x\succ y)\prec \delta z&=\delta x\succ (y\prec z), \\
	\delta x\succ(y\succ z)&=(x\prec y+x\succ y) \succ \delta z.
	\end{align*}
\end{definition}
It is known that from a dendriform algebra many algebras rises like associative, zinbiel, pre-Lie algebras (see \cite{A1}, for more details). In the next we generalize this algebraic structural relation to the case of \textbf{InvDer algebras}.
\begin{proposition}
	An \textbf{InvDer dendriform algebra} $(L,\prec,\succ,\delta)$ satisfying $x\succ y=\prec x=x\diamond y$, for all $x,y \in L$, is an \textbf{InvDer zinbiel algebra}.
\end{proposition}
\begin{proof}
	We have 
	\begin{align*}
	\delta x \diamond (y\diamond z)-(x\diamond y) \diamond \delta z-(y\diamond x)\diamond \delta z &=\delta x \diamond (y\diamond z)-(x\diamond y+y\diamond x)\diamond \delta z \\
	&=\delta x\succ(y\succ z)-(x\prec y+x\succ y)\succ \delta z \\
	&=0,
	\end{align*}
	where $x,y,z \in L$. This completes the proof.
\end{proof}
\begin{proposition}
	Let $(L,\prec,\succ,\delta)$ be an \textbf{InvDer dendriform algebra}. Then $(L,\mu,\delta)$ is axactly an \textbf{InvDer associative algebra}, where 
	\begin{equation*}
	\mu(x,y)=x\succ y+x\prec y, \ \ \ \text{ for } x,y \in L.
	\end{equation*}
\end{proposition}
\begin{proof}
	Let $x,y,z \in L$, so we have
	\begin{align*}
	\mu(\mu(x,y),\delta z)-\mu(\delta x,\mu(y,z))&=\mu(x,y)\succ\delta z+\mu(x,y)\prec\delta z-\delta x \succ \mu(y,z)-\delta x \prec \mu(y,z) \\
	&=(x\succ y)\succ\delta z+(x\prec y)\succ \delta z+(x\succ y)\prec\delta z+(x\prec y)\prec \delta z \\
	&-\delta x\succ(y\succ z)-\delta x\succ(y\prec z)-\delta x\prec(y\succ z)-\delta x\prec(y\prec z) \\
	&=0,
	\end{align*}
	which completes the proof.
\end{proof}
\begin{proposition}
	Let $(L,\prec,\succ,\delta)$ be an \textbf{InvDer dendriform algebra}. Then $(L,\star,\delta)$ is an \textbf{InvDer pre-Lie algebra}, where 
	\begin{equation*}
	x\star y=x\succ y-y\prec x, \ \ \ \forall x,y \in L.
	\end{equation*}
\end{proposition}
\begin{proof}
	The proof is easy to be verified.
\end{proof}
In the next theorem we investigate the well known D.Yau twist in the case of \textbf{InvDer associative algebras}
\begin{theorem}
	Let $(L,\mu)$ be an associative algebra and $\delta \in \textbf{InvDer(L)}$. Then $(L,\mu_{\delta}=\delta \circ \mu,\delta)$ is an \textbf{InvDer associative algebra} if and only if $(L,\mu,\delta)$ is an \textbf{InvDer associative algebra}
\end{theorem}
\begin{proof}
	For any $x,y,z \in L$, we get
	\begin{align*}
	\mu_{\delta}(\mu_{\delta}(x,y),\delta z)-\mu_{\delta}(\delta x,\mu_{\delta}(y,z))&=\mu_{\delta}\circ (\delta \circ \mu(x,y),\delta z)-\mu_{\delta}(\delta x,\delta \circ \mu(y,z)) \\
	&=\delta \circ \mu (\delta \circ \mu(x,y),\delta z)-\delta \circ \mu(\delta x,\delta \circ \mu(y,z)) \\
	&=\mu (\delta^2\circ \mu(x,y),\delta z)+\mu (\delta \circ \mu(x,y),\delta^2z) \\
	&-\mu(\delta^2x,\delta\circ \mu(y,z))-\mu(\delta x,\delta^2\circ \mu(y,z))\\
	&=\mu (\mu(\delta x,\delta y),\delta z)+\mu (\delta \circ \mu(x,y),\delta^2z) \\
	&-\mu (\delta^2x,\delta \circ \mu(y,z))-\mu(\delta x,\mu(\delta y,\delta z))\\
	&=\mu (\delta \circ \mu(x,y),\delta^2z) 
	-\mu (\delta^2x,\delta \circ \mu(y,z))\\
	&=\delta^2 \circ\big(\mu(\mu(x,y),\delta z)-\mu(\delta x,\mu(y,z)) \big) \\
	&=0,
	\end{align*}
	so the assertion is hold.
\end{proof}
Now we generalize the previous theorem to other algebras that we studied before.
\begin{corollary}
	Let $L=(L,\mu)$ be an algebra (not-necessarily associative) and $\delta \in \textbf{InvDer(L)}$. Denoting $L_{\delta}:=(L,\mu_{\delta}=\delta \circ \mu,\delta)$, the following assertions hold:
	\begin{enumerate}
		\item [1-] If $L$ is a Lie algebra, then $L_{\delta}$ is an \textbf{InvDer Lie algebra} if and only if $(L,\mu,\delta)$ is too,\\
		\item [2-] If $L$ is a pre-Lie algebra, then $L_{\delta}$ is an \textbf{InvDer pre-Lie algebra} if and only if $(L,\mu,\delta)$ is too,\\
		\item [3-] If $L$ is a zinbiel algebra, then $L_{\delta}$ is an \textbf{InvDer zinbiel algebra} if and only if $(L,\mu,\delta)$ is too.
	\end{enumerate}
\end{corollary}



\end{document}